\documentclass{amsproc}

\usepackage[T1]{fontenc}
\usepackage{xypic, amssymb, stmaryrd}
\usepackage{url}
\usepackage{array}
\usepackage{caption}
\usepackage{subcaption}
\usepackage[table]{xcolor}

\usepackage{tikz}
\usetikzlibrary{tqft}
\usetikzlibrary{arrows}
\usetikzlibrary{cd}
\tikzset{>=latex}

\usepackage{hyperref}
\usepackage{amsrefs}

\newcommand{\stringdiagram}[1]{\[\begin{tikzpicture}[scale=.5, thick]
#1\end{tikzpicture}\]}

\newcommand{\identity}[2]{\draw (#1,#2+1) -- (#1,#2-1);}
\newcommand{\unit}[2]{\draw (#1,#2+.5) circle [radius=.1]; \draw (#1,#2+.4) -- (#1,#2-1);}
\newcommand{\counit}[2]{\draw (#1,#2-.5) circle [radius=.1]; \draw (#1,#2-.4) -- (#1,#2+1);}

\newcommand{\multiplication}[2] {\draw (#1,#2) -- (#1-1,#2+1); \draw (#1,#2) -- (#1+1,#2+1); \draw (#1,#2) -- (#1,#2-1);}

\newcommand{\comultiplication}[2] {\draw (#1,#2) -- (#1-1,#2-1); \draw (#1,#2) -- (#1+1,#2-1); \draw (#1,#2) -- (#1,#2+1);}
\newcommand{\copairing}[2]{\draw (#1+1,#2-1) arc(0:180:1);}

\newcommand{\equals}[2]{\node at (#1,#2) {$=$};}

\makeatletter
\newcommand*{\da@rightarrow}{\mathchar"0\hexnumber@\symAMSa 4B }
\newcommand*{\xdashrightarrow}[2][]{%
  \mathrel{%
    \mathpalette{\da@xarrow{#1}{#2}{}\da@rightarrow{\,}{}}{}}}
\newcommand*{\da@xarrow}[7]{
\sbox0{$\ifx#7\scriptstyle\scriptscriptstyle\else\scriptstyle\fi#5#1#6\m@th$}
  \sbox2{$\ifx#7\scriptstyle\scriptscriptstyle\else\scriptstyle\fi#5#2#6\m@th$}
  \sbox4{$#7\dabar@\m@th$}
  \dimen@=\wd0 
  \ifdim\wd2 >\dimen@
    \dimen@=\wd2 
  \fi
  \count@=2 %
  \def\da@bars{\dabar@\dabar@}%
  \@whiledim\count@\wd4<\dimen@\do{
    \advance\count@\@ne \expandafter\def\expandafter\da@bars\expandafter{\da@bars\dabar@ }}
  \mathrel{#3}\mathrel{\mathop{\da@bars}\limits \ifx\\#1\\\else _{\copy0} \fi \ifx\\#2\\
   \else  ^{\copy2} \fi}   \mathrel{#4}}
\makeatother

\newcommand{\cat}{\mathcal{C}}
\DeclareMathOperator{\Ob}{Ob}
\DeclareMathOperator{\Mor}{Mor}
\newcommand{\catname}[1]{\mathbf{#1}}
\newcommand{\rel}{\catname{Rel}}

\newcommand{\Span}{\catname{Span}}

\newcommand{\id}{\mathrm{id}}

\newcommand{\idx}{\mathbf{1}}

\newcommand{\suchthat}{\mid}
\newcommand{\Z}{\mathbb{Z}}

\renewcommand{\tilde}{\widetilde}

\renewcommand{\emptyset}{\varnothing}
\newcommand{\power}{\mathcal{P}}

\newtheorem{thm}{Theorem}[section]
\newtheorem{prop}[thm]{Proposition}
\newtheorem{lemma}[thm]{Lemma}

\theoremstyle{definition}
\newtheorem{definition}[thm]{Definition}

\newtheorem{remark}[thm]{Remark}

\numberwithin{equation}{section}

\begin{document}

\title{On Examples and Classification of Frobenius Objects in Rel}
\author{Ivan Contreras}
\address{Department of Mathematics\\
Amherst College\\
31 Quadrangle Drive\\
Amherst, MA 01002}
\email{icontreraspalacios@amherst.edu}
\author{Adele Long}
\author{Sophia Marx}
\address{Department of Mathematics \& Statistics\\
University of Massachusetts, Amherst\\
710 N Pleasant St\\
Amherst, MA 01003}
\email{semarx@umass.edu}
\author{Rajan Amit Mehta}
\address{Department of Mathematics \& Statistics\\
Smith College\\
44 College Lane\\
Northampton, MA 01063}
\email{AdeleLRLong@gmail.com}
\email{rmehta@smith.edu}


\begin{abstract}
We give some new examples of Frobenius objects in the category of sets and relations $\rel$. One example is a groupoid with a twisted counit. Another example is the set of conjugacy classes of a group. We also classify Frobenius objects in $\rel$ with two or three elements, and we compute the associated surface invariants using the partition functions of the corresponding TQFTs. 
\end{abstract}

\maketitle

\section{Introduction}

A basic result in topological quantum field theory (TQFT) is the correspondence between $2$-dimensional oriented TQFTs and commutative Frobenius algebras \cites{abrams, dijkgraaf:thesis}. This result can be placed in a more general framework by defining a (commutative) Frobenius object in a (symmetric) monoidal category  $\cat$ (see, e.g. \cite{kock-book}). Then a proof of the above correspondence can be reinterpreted as a proof that the $2$-dimensional oriented cobordism category is isomorphic to the free symmetric monoidal category generated by one commutative Frobenius object. The upshot of this is that the study of Frobenius objects in any symmetric monoidal category $\cat$ has topological significance.

In this paper, we study Frobenius objects in the category $\rel$, whose objects are sets and whose morphisms are relations of sets. Such structures appeared in \cite{hcc} (also see \cite{heunen-vicary:book}), where it was shown that special dagger Frobenius objects in $\rel$ are in correspondence with groupoids. This result was extended in \cite{Mehta-Zhang}, where it was shown that a Frobenius object in $\rel$ can be encoded in a simplicial set equipped with an automorphism of the set of $1$-simplices, satisfying certain properties.

Frobenius objects in the category of spans (which is closely related to $\rel$) have also been recently considered. In \cite{stern:2segal}, it was shown that symmetric Frobenius objects in $\Span$ that are coherent (in the $2$-categorical sense) are in correspondence with cyclic $2$-Segal sets. An analogous correspondence at the $1$-categorical level was given in \cite{CKM}. There is a symmetric monoidal functor $\Span \to \rel$, so any Frobenius object in $\Span$ gives rise to one in $\rel$. 

It turns out \cites{li-bland-weinstein, CKM} that $\Span$ can be viewed as a set-theoretic model for the Wehrheim-Woodward symplectic category \cite{weinstein:ww, ww}. Thus the study of Frobenius objects in $\rel$ provides a relatively simple setting to better understand TQFT with values in the symplectic category. Since symplectic groupoids are examples of Frobenius objects in the symplectic category, we expect this direction of research to shed light on the role of symplectic groupoids as reduced phase spaces of the Poisson sigma model, a $2$-dimensional topological field theory \cite{Contreras2015}. There may also be connections to higher-dimensional TQFT, such as Dikjgraaf-Witten theory \cite{dijgraafwitten}.

The main results of this paper are as follows:
\begin{itemize}
    \item We describe new examples of Frobenius objects in $\rel$. One is a generalization of the groupoid example which allows for a twist that can spoil the special and dagger properties. As we will see, the twists are necessary to obtain nontrivial topological invariants. Another example is the set of conjugacy classes of a group. These examples can have multivalued multiplication relations.
    \item We completely classify Frobenius objects in $\rel$ with two or three elements. We find that, up to isomorphism, there are $5$ Frobenius objects in $\rel$ with two elements, and there are $25$ Frobenius objects in $\rel$ with three elements. 
    This suggests that most examples are not groupoids, since only $5$ of the three-element examples are groupoids.
    \item For all of the two- and three-element Frobenius objects in $\rel$, we compute the partition function, which encodes the associated topological invariants for closed orientable surfaces. These partition functions can be viewed as logical propositions dependent on the genus of the surface.
\end{itemize}

Our low-cardinality classification was done completely by hand, but it involves some simplifications and systematic calculations that should allow for an extension to higher cardinalities with some computer assistance. 
\subsection*{Acknowledgements}
We would like to thank Pavel Mnev and Walker Stern for stimulating discussions on topics related to this paper. I.C. thanks the Amherst College Provost and Dean of the Faculty’s Research Fellowship (2021-2022).
\section{Frobenius objects and TQFT}

In this section, we review the definition of a Frobenius object in a (symmetric) monoidal category $\mathcal{C}$, and we briefly explain the relationship to $2$-dimensional TQFT.

\subsection{Frobenius objects in a monoidal category}

Let $\mathcal{C}$ be a monoidal category with monoidal unit $1$ and monoidal product $\otimes$. For simplicity, we will assume that $\mathcal{C}$ is strict monoidal, though everything in this section works more generally, with appropriate modification.

\begin{definition}
A \emph{Frobenius object} in $\mathcal{C}$ is an object $X \in \Ob(\mathcal{C})$ equipped with morphisms 

\begin{itemize}
\item $\eta: 1 \to X$ (Unit) 
\item $\mu: X \otimes X \to X$ (Multiplication)
\item $\varepsilon: X \to 1$ (Counit)
\end{itemize}
satisfying the following axioms:
\begin{enumerate}
    \item Unitality: $\mu \circ (\idx \otimes \eta) = \mu \circ (\eta \otimes \idx) = \idx$
    \item Associativity: $\mu \circ (\idx \otimes \mu) = \mu \circ (\mu \otimes \idx)$
    \item Nondegeneracy: There exists $\beta: 1 \to X \otimes X$ such that $(\varepsilon \otimes \idx) \circ (\mu \otimes \idx) \circ (\idx \otimes \beta) = (\idx \otimes \varepsilon) \circ (\idx \otimes \mu) \circ (\beta \otimes \idx) = \idx$.
\end{enumerate}
\end{definition}
In the case where $\mathcal{C}$ is the monoidal category of vector spaces with the tensor product, one recovers the notion of \emph{Frobenius algebra}.

It can be helpful to use string diagrams to describe morphisms built out of the structure morphisms for a Frobenius object. We denote the unit, multiplication, and counit by the following diagrams, read from top to bottom.
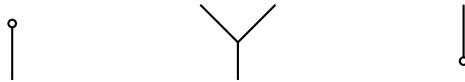
\stringdiagram{\unit{0}{0} \multiplication{6}{0} \counit{12}{0}}
It can be proven that $\beta$ in the nondegeneracy condition is unique. It is denoted as follows.
\stringdiagram{\copairing{0}{0}}
The equations in the axioms can be rewritten using string diagrams; see Figure \ref{fig:axioms}.
\begin{figure}
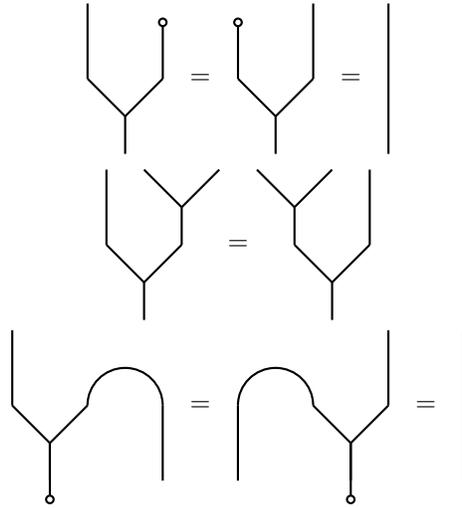

\stringdiagram{
\begin{scope} 
\unit{-2}{1}
\identity{-4}{1}
\multiplication{-3}{-1}
\equals{-1}{0}
\unit{0}{1}
\identity{2}{1}
\multiplication{1}{-1}
\equals{3}{0}
\identity{4}{1}
\identity{4}{-1}
\end{scope}}
\stringdiagram{
\begin{scope}[shift={(8,0)}] 
\identity{0}{1}
\multiplication{2}{1}
\multiplication{1}{-1}
\equals{3.5}{0}
\multiplication{5}{1}
\identity{7}{1}
\multiplication{6}{-1}
\end{scope}}
\stringdiagram{\begin{scope}[shift={(20,-1)}]
 \identity{-1}{2}
\copairing{2}{2}
\multiplication{0}{0}
\counit{0}{-1}
\identity{3}{0}
\equals{4}{1}
\copairing{6}{2}
\identity{9}{2}
\identity{5}{0}
\multiplication{8}{0}
\counit{8}{-1}
\equals{10}{1}
\identity{11}{2}
\identity{11}{0}
\end{scope}
}
    \caption{The unitality, associativity, and nondegeneracy axioms via string diagrams.}
    \label{fig:axioms}
\end{figure}

Given a Frobenius object, we may define a comultiplication $\delta: X \to X \otimes X$ as follows.
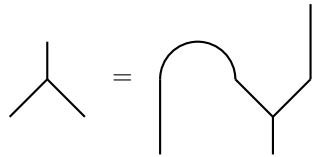
\stringdiagram{
\comultiplication{0}{0}
\equals{2}{0}
\copairing{4}{1}
\identity{7}{1}
\identity{3}{-1}
\multiplication{6}{-1}
}
One can then use the axioms to show that the comultiplication is counital and coassociative. Another nice exercise for the reader is to show that the \emph{Frobenius condition} in Figure \ref{fro} is satisfied. 
\begin{figure}[h]
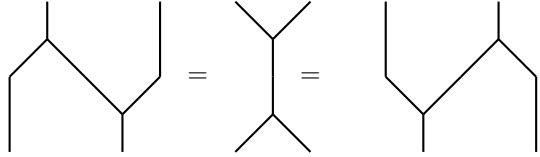


\stringdiagram{
\comultiplication{0}{0}
\identity{-1}{-2}
\multiplication{2}{-2}
\identity{3}{0}
\equals{4}{-1}
\multiplication{6}{0}
\comultiplication{6}{-2}
\equals{7}{-1}
\identity{9}{0}
\multiplication{10}{-2}
\comultiplication{12}{0}
\identity{13}{-2}

}
\caption{The Frobenius condition.}
\label{fro}
\end{figure}

We note that the definition of Frobenius object can be stated in different equivalent ways; see, for example, \cites{kock-book,Mehta-Zhang}.

\subsection{Relation to TQFT}\label{sec:tqft}

If $\mathcal{C}$ is a symmetric monoidal category, then we can define commutative Frobenius objects. A well-known result \cites{abrams,dijkgraaf:thesis} is that commutative Frobenius objects in $\mathcal{C}$ correspond to $\mathcal{C}$-valued $2$-dimensional TQFTs, i.e.\ symmetric monoidal functors from the 2D oriented cobordism category to $\mathcal{C}$. We refer the reader to \cite{kock-book} for details about cobordism categories and a concise proof of this result.

In particular, a commutative Frobenius object gives invariants of closed orientable surfaces. The invariants appear as the output of the \emph{partition function} $Z$ of the theory, which takes closed orientable surfaces as input and takes values in the commutative monoid $\Mor_\mathcal{C}(1,1)$. The value of the partition function on the closed orientable surface $\Sigma_g$ of genus $g$ is explicitly given by the formula
\begin{equation}\label{eqn:partition}
Z(\Sigma_g)= \varepsilon \circ (\mu \circ \delta)^g \circ \eta.
\end{equation}
The corresponding string diagrams are in Figure \ref{fig:invariants}.
\begin{figure}
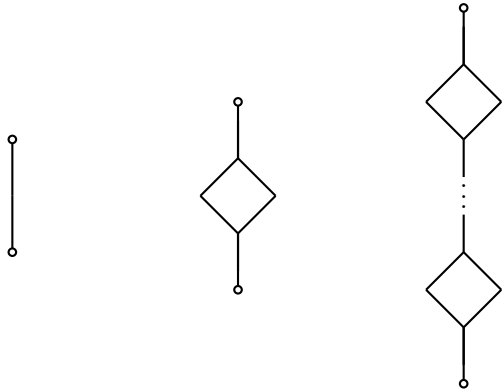

\stringdiagram{
\begin{scope}
\unit{0}{1}
\counit{0}{-1}
\end{scope}

\begin{scope}[shift={(6,0)}]
\unit{0}{2}
\comultiplication{0}{1}
\multiplication{0}{-1}
\counit{0}{-2}
\end{scope}

\begin{scope}[shift={(12,3.5)}]
\unit{0}{1}
\comultiplication{0}{0}
\multiplication{0}{-2}
\node at (0,-3.3) [] (f) {$\vdots$};
\comultiplication{0}{-5}
\multiplication{0}{-7}
\counit{0}{-8}
\end{scope}
}
    \caption{Invariants associated to the sphere, torus, and higher genus surfaces.}
    \label{fig:invariants}
\end{figure}

\section{Frobenius objects in Rel}\label{sec:frob}
In this section, we briefly review the category of relations $\rel$, and we then describe Frobenius objects in $\rel$ associated to groupoids and conjugacy classes.

\subsection{The category of relations}\label{sec:rel}
The objects of $\rel$ are sets. A morphism from a set $X$ to a set $Y$ is a relation, i.e.\ a subset of $X \times Y$. Given relations $R \subseteq X \times Y$ and $S \subseteq Y \times Z$, the composition $S \circ R \subseteq X \times Z$ is defined as
\[ S \circ R = \{ (x,z) \suchthat (x,y) \in R \mbox{ and } (y,z) \in S \mbox{ for some } y \in Y\}.\]
The identity morphism from $X$ to $X$ is the diagonal relation:
\[ \idx = \{ (x,x) \suchthat x \in X\}.\]
The Cartesian product gives $\rel$ the structure of a symmetric monoidal category. The monoidal unit is the one-point set, which we denote as $\{\bullet\}$.

It can sometimes be convenient to think of a relation $R \subseteq X \times Y$ as a generalized map that can take any number of values. To make this idea precise, we define a map $\tilde{R}: X \to \power(Y)$, where $\power(Y)$ is the power set of $Y$, given by
\[ \tilde{R}(x) = \{y \in Y \suchthat (x,y) \in R\}.\]
This gives a one-to-one correspondence between relations $R \subseteq X \times Y$ and maps $\tilde{R}: X \to \power(Y)$. In terms of the latter, composition is given by
\begin{equation} \label{eqn:compositiontilde}
\widetilde{S \circ R}(x) = \bigcup_{y \in \tilde{R}(x)} \tilde{S}(y).
\end{equation}

The data of a Frobenius object in $\rel$ consists of a set $X$, together with subsets $\eta \subseteq \{\bullet\} \times X \cong X$, $\varepsilon \subseteq X \times \{\bullet\} \cong X$, and $\mu \subseteq X \times X \times X$. In the remainder of the paper, we will view $\eta$ and $\varepsilon$ as subsets of $X$, and we will usually consider $\mu$ via the corresponding map $\tilde{\mu}: X \times X \to \power(X)$.

When checking that the data $(X,\eta, \varepsilon, \tilde{\mu})$ satisfies the axioms of a Frobenius object, unitality and associativity can be checked directly, but the nondegeneracy axiom is a bit clumsy as given.
In \cite{Mehta-Zhang}, a useful characterization of the nondegeneracy axiom is given. The condition says that, for every $x \in X$, there exists a unique $y \in X$ such that $\tilde{\mu}(x,y) \cap \varepsilon$ is nonempty. If this condition is satisfied, the correspondence $x \mapsto y$ determines a bijection $\hat{\alpha}: X \to X$.

From \cite{Mehta-Zhang}, we also have that, if $(X,\eta, \varepsilon, \tilde{\mu})$ is a Frobenius object in $\rel$, then the comultiplication, viewed as a map $\tilde{\delta}: X \to \power(X \times X)$, can be expressed in terms of $\tilde{\mu}$ and the bijection $\hat{\alpha}$, as follows: 
\begin{equation}\label{eqn:comultiplication}
\tilde{\delta}(x) = \{(\hat{\alpha}(y), z) \suchthat y \in X, z \in \tilde{\mu}(y,x)\}.
\end{equation}

\subsection{Disjoint unions} \label{sec:disjoint}

Let $(X, \eta_X, \varepsilon_X, \tilde{\mu}_X)$ and $(Y, \eta_Y, \varepsilon_Y, \tilde{\mu}_Y)$ be Frobenius objects in $\rel$. Then we may define a Frobenius structure on the disjoint union $X \sqcup Y$ as follows:
\begin{itemize}
    \item $\eta_{X \sqcup Y} = \eta_X \sqcup \eta_Y$,
    \item $\varepsilon_{X \sqcup Y} = \varepsilon_X \sqcup \eta_Y$, and
    \item the multiplication is in block form, i.e.\ for $x,x' \in X$ and $y, y' \in Y$, we set 
    \begin{align*}
        \tilde{\mu}_{X \sqcup Y} (x,x') &= \tilde{\mu}_X(x,x'), \\
        \tilde{\mu}_{X \sqcup Y} (y,y') &= \tilde{\mu}_Y(y,y'), \\
        \tilde{\mu}_{X \sqcup Y} (x,y) &= \tilde{\mu}_{X \sqcup Y} (y,x) = \emptyset.
    \end{align*}
\end{itemize}
It is fairly straightforward to check that the axioms of a Frobenius object hold for $X \sqcup Y$ as a consequence of the fact that they hold for $X$ and $Y$.

\subsection{Example: groupoids} \label{sec:groupoids}
In \cite{hcc} it was shown that groupoids can be seen as Frobenius objects in $\rel$. Conversely, they showed that any Frobenius object in $\rel$ satisfying the extra properties of being \emph{special} and \emph{dagger} is associated to a groupoid. Here, we briefly review this correspondence before introducing the additional possibility of a twisted counit.

Recall that a \emph{groupoid} is a small category where all morphisms are invertible. In more concrete terms, a groupoid $G$ consists of sets $G_0, G_1$, equipped with \emph{source} and \emph{target} maps $s,t: G_1 \to G_0$ and a \emph{multiplication} operation $(g,h) \mapsto g \cdot h$, defined when $s(g) = t(h)$ for $g,h \in G_1$, such that
\begin{enumerate}
    \item $s(g \cdot h) = s(h)$ and $t(g \cdot h) = t(g)$ for $g,h \in G_1$ such that $s(g) = t(h)$,
    \item $(g\cdot h) \cdot k = g \cdot (h \cdot k)$ for all $g,h,k \in G_1$ such that $s(g) = t(h)$ and $s(h) = t(k)$,
    \item there exists an \emph{identity} map $e: G_0 \to G_1$ such that $s \circ e = t \circ e = \id$ and $g \cdot e(s(g)) = e(t(g)) \cdot g = g$ for all $g \in G_1$,
    \item there exists an \emph{inverse map} $G_1 \to G_1$, $g \mapsto g^{-1}$, such that $s(g^{-1}) = t(g)$, $t(g^{-1}) = s(g)$, $g \cdot g^{-1} = e(t(g))$, and $g^{-1} \cdot g = e(s(g))$ for all $g \in G_1$.
\end{enumerate}
As with groups, the identity and inverse maps for a groupoid are unique.

Given a groupoid $G$, the corresponding Frobenius object in $\rel$ given in \cite{hcc} is as follows:
\begin{itemize}
    \item $X=G_1$,
    \item $\eta = e(G_0)$,
    \item $\tilde{\mu}(g,h) = \{g \cdot h\}$,
    \item $\varepsilon = e(G_0)$.
\end{itemize}
Associativity and unitality hold as a result of conditions (2) and (3) above. Nondegeneracy holds as a result of condition (4); in particular, the associated bijection $\hat{\alpha}$ is given by $\hat{\alpha}(g) = g^{-1}$.

This example can be generalized by introducing the possibility of ``twisting'' the counit as follows. Let $\sigma: G_0 \to G_1$ be a section of $t$, and set $\varepsilon = \sigma(G_0)$. With this choice of counit, one can see that the nondegeneracy condition still holds, with $\hat{\alpha}(g) = g^{-1} \cdot \sigma(t(g))$. We will see in Section \ref{sec:invariants} that this generalization is necessary to obtain nontrivial topological invariants.

\subsection{Example: conjugacy classes}\label{subsec:Conjugacy}
Let $G$ be a group, and let $X$ be the set of conjugacy classes of $G$. Set $\eta = \varepsilon = \{[e]\}$, and let $\tilde{\mu}: X \times X \to \power(X)$ be given by
\[ \tilde{\mu}(C_1,C_2) = \{ C_3 \suchthat g_1g_2 = g_3 \mbox{ for some } g_i \in C_i\}. \]
Then $X$ is a Frobenius object in $\rel$. The associated bijection $\hat{\alpha}$ is given by $\hat{\alpha}([g]) = [g^{-1}]$. We note that nondegeneracy holds due to the not-completely-obvious fact that, for any $g \in G$, the conjugacy class of $g^{-1}$ is the set of inverses of elements of the conjugacy class of $g$.
  
This is in fact an example of a commutative Frobenius object in $\rel$, since $gh$ and $hg$ are conjugate to each other for all $g,h \in G$: 
 \[g^{-1} (gh) g=hg. \]
It is also an example in which the multiplication can be multi-valued. This is notable, because all examples of Frobenius objects in $\rel$ given in \cite{hcc} and \cite{Mehta-Zhang} have (at most) single-valued multiplication.

In Section \ref{sec:3element}, we will see the set of conjugacy classes of $S_3$ appearing as an explicit example.

\section{Topological invariants}\label{sec:invariants}

Any commutative Frobenius object in $\rel$ gives rise to surface invariants (see Section \ref{sec:tqft}). These invariants take values in $\Mor_\rel(\{\bullet\},\{\bullet\}) \cong \power(\{\bullet\} \times \{\bullet\}) \cong \power(\{\bullet\})$. It can be useful to identify the values with Booleans, where $F$ is identified with $\emptyset$ and $T$ is identified with $\{\bullet\}$. Under this identification, the monoidal operation, arising from composition of relations, corresponds to the AND operator.

Thus, the partition function associated to a commutative Frobenius object in $\rel$, given by \eqref{eqn:partition}, can be viewed as a Boolean-valued function of the genus $g$.

Let $(X, \eta, \varepsilon, \tilde{\mu})$ be a commutative Frobenius object in $\rel$, with comultiplication and associated bijection $\hat{\alpha}$ as described in Section \ref{sec:rel}. Let $S = \mu \circ \delta$. Using \eqref{eqn:compositiontilde} and \eqref{eqn:comultiplication}, we have
\begin{equation}\label{eqn:S}
    \begin{split}
        \tilde{S}(x) &= \widetilde{\mu \circ \delta}(x) \\
        &= \bigcup_{\phi \in \tilde{\delta}(x)} \tilde{\mu}(\phi) \\
        &= \bigcup_{y \in X} \bigcup_{z \in \tilde{\mu}(y,x)} \tilde{\mu}(\hat{\alpha}(y),z).
    \end{split}
\end{equation}
This general formula may not seem very enlightening, but in practice it often leads to a relatively nice description of $S$. From this, we can determine $S^g \circ \eta$, which, as a relation from $\{\bullet\}$ to $X$, can be viewed as a subset of $X$.
Then the value of the partition function $Z(\Sigma_g) = \varepsilon \circ S^g \circ \eta$ is $T$ if $(S^g \circ \eta) \cap \varepsilon$ is nonempty and $F$ if $(S^g \circ \eta) \cap \varepsilon = \emptyset$.

\begin{prop} \label{prop: partition_group}Let $G$ be a finite abelian group. For fixed $\omega \in G$, let $X=G$ be the Frobenius object in $\rel$ associated to $G$ with counit $\varepsilon = \{\omega\}$ (see Section \ref{sec:groupoids}). Then the associated partition function is given by
\begin{equation}
Z(\Sigma_g) = \begin{cases}
T \mbox{ if } (g-1)\omega=0 \\
F \mbox{ otherwise. }
\end{cases}
\end{equation}
\end{prop}
\begin{proof}
In this situation, we have the single-valued multiplication $\tilde{\mu}(x,y) = \{x+y\}$ and $\hat{\alpha}(x) = \omega - x$. Then, from \eqref{eqn:S} we have
\begin{equation}
\begin{split}
\tilde{S}(x) &= \bigcup_{y \in G} \{\hat{\alpha}(y) + y + x\} \\
&= \bigcup_{y \in G} \{\omega + x\} \\
&= \{\omega + x\}.
\end{split}
\end{equation}  
We then obtain 
\begin{align*}
    \widetilde{S^g}(x) &= \{g\omega + x\}, \\
    S^g \circ \eta &= \{g\omega\},
\end{align*}
so $Z(\Sigma_g) = T$ if and only if $g\omega = \omega$.
\end{proof}

The following result will be useful later in computing partition functions for the low-cardinality cases.
\begin{lemma}\label{lemma:intersection}
Let $(X, \eta, \varepsilon, \tilde{\mu})$ be a commutative Frobenius object in $\rel$. If $\eta \cap \varepsilon \neq \emptyset$, then the associated partition function is given by $Z(\Sigma_g) = T$ for all $g$.
\end{lemma}

\begin{proof}
    Let $e \in X$ be such that $e \in \eta \cap \varepsilon$. From the unitality axiom, we have that $\tilde{\mu}(e,e) = \{e\}$. It follows that $\hat{\alpha}(e) = e$, so we see from \eqref{eqn:S} that $e \in \tilde{S}(e)$. We then have that $e \in (S^g \circ \eta) \cap \varepsilon$ (including the case $g=0$), and hence $Z(\Sigma_g) = T$.
\end{proof}

\begin{remark}\label{remark:logical}
Because the partition function associated to a Frobenius object in $\rel$ is Boolean-valued, it can be viewed as a logical proposition depending on the variable $g$. This point of view can lead to a more intuitively clear description of the partition function. For example, in the situation of Proposition \ref{prop: partition_group} the partition function is given by the proposition ``$g \equiv 1 \pmod{|\omega|}$''. In the situation of Lemma \ref{lemma:intersection}, the partition function is given by the proposition ``True''.
\end{remark}

\section{The two element case}\label{sec:2element}
\subsection{The classification}
In this section, we classify the Frobenius objects in $\rel$ where $X = \{a,b\}$ is a two element set. Recall that a Frobenius structure (in $\rel$) on $X$ is given by three pieces of data: the unit $\eta \subseteq X$, the counit $\varepsilon \subseteq X$, and the multiplication $\mu \subseteq X \times X \times X$. When $X$ has two elements there are \emph{a priori} $2^2=4$ possibilities for each of $\eta$ and $\varepsilon$ and $2^8 = 256$ possibilities for $\mu$, giving a total of $2^{12} = 4096$ possibilities for the data. As we will see, the axioms for a Frobenius object impose strong restrictions, leaving only five possibilities up to isomorphism.

As in Section \ref{sec:frob}, we describe $\mu$ via the associated map $\tilde{\mu}: X \times X \to \power(X)$, which can be represented by a multiplication table
\vskip 3mm
\begin{center}
\begin{tabular}{ | c | c | } 
\hline
$\tilde{\mu}(a,a)$ & $\tilde{\mu}(a,b)$ \\
\hline
$\tilde{\mu}(b,a)$ & $\tilde{\mu}(b,b)$ \\
\hline
\end{tabular}
\end{center}
\vskip 3mm
This is similar to the multiplication tables that students encounter when learning group theory, except that, because $\tilde{\mu}$ takes values in the power set $\power(X)$, the entries in the table are subsets of $X$, rather than elements of $X$.

We begin by considering the different possibilities for the unit relation $\eta$. Up to isomorphism, there are three different cases: $\eta = \emptyset, \{a\}, \{a,b\}$. In each case, we can then see what restrictions on $\mu$ arise from unitality.

One side of the unitality axiom requires that, for each $x$, there exists $e \in \eta$ such that $x \in \tilde{\mu}(x,e)$, and there does not exist any $e \in \eta$ such that $y \in \tilde{\mu}(x,e)$ for $y \neq x$. The other side of the unitality axiom requires a similar condition with the inputs of $\tilde{\mu}$ reversed. This condition immediately rules out the case where $\eta = \emptyset$. 

In the case where $\eta = \{a\}$, unitality determines three of the four entries in the multiplication table:
\vskip 3mm
\begin{center}
\begin{tabular}{ | c | c |} 
\hline
$\{a\}$  & $\{b\}$ \\ 
\hline
$\{b\}$ &  \\ 
\hline
\end{tabular}
\end{center}
\vskip 3mm
Unitality doesn't impose any constraints on the remaining entry $\tilde{\mu}(b,b)$.

In the case where $\eta = \{a,b\}$, unitality completely determines the multiplication table:
\vskip 3mm
\begin{center}
\begin{tabular}{ | c | c | } 
\hline
$\{a\}$  & $\emptyset$ \\ 
\hline
$\emptyset$ & $\{b\}$  \\ 
\hline
\end{tabular}
\end{center}
\vskip 3mm

Next, we consider the different possibilities for the counit $\varepsilon$. From \cite{Mehta-Zhang}, we know that the unit and counit relations must have the same number of elements. This gives us a small number of cases, and in each case we can see what restrictions on $\mu$ arise from nondegeneracy. The nondegeneracy condition (as described in Section \ref{sec:rel}) requires that, in each row or column, there exists a unique entry containing an element of $\varepsilon$.
\begin{itemize}
    \item $\eta = \{a\}$, $\varepsilon = \{a\}$. 
    \newline
    Here, nondegeneracy requires that $\tilde{\mu}(b,b)$ contains $a$,  leaving two possibilities for the multiplication table:    
\vskip 3mm
\begin{center}
\begin{tabular}{ | c | c |} 
\hline
$\{a\}$  & $\{b\}$ \\ 
\hline
$\{b\}$ & $\{a\}$ \\ 
\hline
\end{tabular}
\hspace{1cm}
\begin{tabular}{ | c | c | } 
\hline
 $\{a\}$  & $\{b\}$ \\ 
\hline
 $\{b\}$ & $\{a,b\}$ \\ 
\hline
\end{tabular}
\end{center}
\vskip 3mm

    \item $\eta = \{a\}$, $\varepsilon = \{b\}$. Here, nondegeneracy requires that $\tilde{\mu}(b,b)$ does not contain $b$, leaving two possibilities for the multiplication table:
\vskip 3mm

\begin{center}
\begin{tabular}{ | c | c |} 
\hline
 $\{a\}$  & $\{b\}$ \\ 
\hline
 $\{b\}$ & $\{a\}$ \\ 
\hline
\end{tabular}
\hspace{1cm}
\begin{tabular}{ | c | c | } 
\hline
 $\{a\}$  & $\{b\}$ \\ 
\hline
 $\{b\}$ & $\emptyset$ \\ 
\hline
\end{tabular}
\end{center}
\vskip 3mm

    \item $\eta = \{a,b\}$, $\varepsilon = \{a,b\}$. As we saw above, the multiplication table is already determined by the unit axiom, but we observe that nondegeneracy is also satisfied in this situation.
\end{itemize}

We are now left with only five possibilities (two of which have the same multiplication but different counits), and it remains to check associativity in each case. For cases 1, 3, and 5 in Figure \ref{fig:2elements}, we can recognize the multiplication table as being that of a groupoid. In the remaining two cases, we can systematically check associativity in the following way. For each element of $X \times X \times X$, we compute the image in $\power(X)$ under both sides of the associativity equation, and see whether the two images are equal. Here we use the rule for compositions in \eqref{eqn:compositiontilde}. 

We find that that associativity holds in both remaining cases, giving us the following result.

\begin{thm}\label{thm:class_2elem}
Up to isomorphism, there are five Frobenius objects in $\rel$ with two elements. They are listed in Table \ref{fig:2elements}.
\end{thm}

\begin{table}[!ht]
\begin{tabular}{|c|c|c|c|c|}
\hline
\textbf{Case}&\textbf{Unit} & \textbf{Counit} & \textbf{Multiplication}& \textbf{Partition function} \\ \hline
& & & &\\
1& $\{a\}$& $\{a\}$ &
\begin{tabular}{ | c | c |  } 
\hline
$\{a\}$  & $\{b\}$ \\ 
\hline
$\{b\}$ & $\{a\}$\\ 
\hline
\end{tabular} & True
\\
[5ex]
\hline
& & & &\\
2& $\{a\}$& $\{a\}$ &
\begin{tabular}{ | c | c |  } 
\hline
$\{a\}$  & $\{b\}$ \\ 
\hline
$\{b\}$ & $\{a,b\}$\\ 
\hline
\end{tabular} & True
\\
[5ex]
\hline
& & & &\\
3& $\{a\}$& $\{b\}$ &
\begin{tabular}{ | c | c |  } 
\hline
$\{a\}$  & $\{b\}$ \\ 
\hline
$\{b\}$ & $\{a\}$\\ 
\hline
\end{tabular} 
& $g$ is odd
\\
[5ex]
\hline
& & & &\\

4& $\{a\}$& $\{b\}$ &
\begin{tabular}{ | c | c |  } 
\hline
$\{a\}$  & $\{b\}$ \\ 
\hline
$\{b\}$ & $\emptyset$\\ 
\hline
\end{tabular} 
& $g=1$
\\
[5ex]
\hline
& & & &\\

5& $\{a,b\}$& $\{a,b\}$ &
\begin{tabular}{ | c | c |  } 
\hline
$\{a\}$  & $\emptyset$ \\ 
\hline
$\emptyset$ & $\{b\}$\\ 
\hline
\end{tabular} 
& True
\\
[5ex]
\hline

\end{tabular}
    \caption{Frobenius objects in $\rel$ with two elements and their partition functions, given as logical propositions.}
    \label{fig:2elements}

\end{table}

\begin{remark}
As mentioned above, cases 1, 3, and 5 in Figure \ref{fig:2elements} are Frobenius objects corresponding to groupoids. Cases 1 and 3 correspond to the group $\Z_2$ with the two possible choices of counits, and case 5 corresponds to the trivial groupoid with two objects. Case 4 also has an interesting interpretation as the cohomology of $S^2$ (see \cite{Mehta-Zhang}). At the moment, we do not have a conceptual interpretation of case 2.
\end{remark}

\subsection{Calculation of partition functions}

One can immediately see from the multiplication tables in Figure \ref{fig:2elements} that all five two-element Frobenius objects in $\rel$ are commutative, so they correspond to $\rel$-valued TQFTs, and we can ask what the associated partition functions are.

From Lemma \ref{lemma:intersection}, we have that the partition function for cases 1, 2, and 5 (when given as a logical proposition, c.f.\ Remark \ref{remark:logical}) is ``True''.

For case 3, Proposition \ref{prop: partition_group} applies. Since the counit element $b$ is of order $2$, the partition function is ``$g$ is odd''.

For the remaining case 4, we can calculate the partition function directly, following the procedure outlined in Section \ref{sec:invariants}. In this case, we have $\hat{\alpha}(a) = b$ and $\hat{\alpha}(b) = a$. From \eqref{eqn:S}, we compute that $\tilde{S}(a) = \{b\}$ and $\tilde{S}(b) = \emptyset$. From this we see that $Z(\Sigma_0) = F$, $Z(\Sigma_1) = T$, and $Z(\Sigma_g) = F$ for $g \geq 2$. Thus, as a logical proposition, the partition function is ``$g=1$''.

These results are listed in the final column of Table \ref{fig:2elements}.

\section {The three element case}\label{sec:3element}

In this section, we classify the Frobenius objects in $\rel$ where $X = \{a,b,c\}$ is a three element set. Although there is more work involved, our approach is essentially the same as in Section \ref{sec:2element}, where we first impose unitality, then nondegeneracy, and then associativity. Note that we often omit calculations that are similar to others that have been already been presented.

Up to isomorphism, we have the following cases for the unit: $\eta=\{a\}, \{a,b\}$, and $\{a,b,c\}$. As in Section \ref{sec:2element}, unitality rules out the case $\eta = \emptyset$. 

\subsection{The case $\eta = \{a\}$}

In this case, unitality partially determines the multiplication table as follows.
\begin{center}
\begin{tabular}{ | c | c | c |} 
\hline
$\{a\}$  & $\{b\}$& $\{c\}$ \\ 
\hline
$\{b\}$ & & \\ 
\hline
$\{c\}$ & & \\ 
\hline
\end{tabular}
\end{center}
\vskip 3mm
Up to isomorphism, there are two possible counits: $\varepsilon = \{a\}, \{b\}$. 

\subsubsection{{The case $\varepsilon = \{a\}$}}
In this case, nondegeneracy restricts the remaining entries to the two following scenarios:
\vskip 3mm
    \begin{center}
    \begin{tabular}{ | c | c | c |} 
\hline
$\{a\}$  & $\{b\}$& $\{c\}$ \\ 
\hline
$\{b\}$ & includes $a$ & does not include $a$ \\ 
\hline
$\{c\}$ & does not include $a$ & includes $a$\\ 
\hline
\end{tabular}
\vskip 3mm
\begin{tabular}{ | c | c | c |} 
\hline
$\{a\}$  & $\{b\}$& $\{c\}$ \\ 
\hline
$\{b\}$ & does not include $a$ & includes $a$ \\ 
\hline
$\{c\}$ & includes $a$ & does not include $a$\\ 
\hline
\end{tabular}
\end{center}
\vskip 3mm
This still leaves many possibilities, but we can make the situation more tractable by considering each possibility for $\tilde{\mu}(b,b)$ and seeing what restrictions on the other entries arise from associativity.
\begin{itemize}
    \item $\tilde{\mu}(b,b) = a$. In this case, 
    \begin{equation*}
    \begin{split}
    \widetilde{\mu \circ (\mu \times \idx)}(b,b,c) &= \bigcup_{x \in \tilde{\mu}(b,b)} \tilde{\mu}(x,c) \\
    &= \tilde{\mu}(a,c) = \{c\},
    \end{split}
    \end{equation*}
    whereas
    \begin{equation*}
    \widetilde{\mu \circ (\idx \times \mu)}(b,b,c) = \bigcup_{x \in \tilde{\mu}(b,c)} \tilde{\mu}(b,x),
    \end{equation*}
    so associativity requires that $\tilde{\mu}(b,c) = \{c\}$. Similarly,
    \begin{equation*}
    \widetilde{\mu \circ (\mu \times \idx)}(c,b,b) = \bigcup_{x \in \tilde{\mu}(c,b)} \tilde{\mu}(x,b),    
    \end{equation*}
    whereas
    \begin{equation*}
    \widetilde{\mu \circ (\idx \times \mu)}(c,b,b) = \{c\},    
    \end{equation*}
    so associativity requires that $\tilde{\mu}(c,b) = \{c\}$.

    Next, we consider
        \begin{equation*}
    \widetilde{\mu \circ (\mu \times \idx)}(b,c,c) = \bigcup_{x \in \tilde{\mu}(b,c)} \tilde{\mu}(x,c) = \tilde{\mu}(c,c)    
    \end{equation*}
    and
    \begin{equation*}
    \widetilde{\mu \circ (\idx \times \mu)}(b,c,c) = \bigcup_{x \in \tilde{\mu}(c,c)} \tilde{\mu}(b,x).
    \end{equation*}
    Since $\tilde{\mu}(c,c)$ includes $a$ in this scenario, we see that the latter expression contains $\tilde{\mu}(b,a) = \{b\}$. So associativity requires $\{a,b\} \subseteq \tilde{\mu}(c,c)$. This leaves us with cases 1 and 2 in Table \ref{tab:3elements_inv_1}, where we can check that associativity holds in full.
    \item $\tilde{\mu}(b,b)=\{a,b\}$. In this case, a similar series of calculations lead to cases 3 and 4.
    \item $\tilde{\mu}(b,b)=\{a,c\}$. Here, there are four cases, two of which are isomorphic (under exchange of $b$ and $c$) to cases 1 and 3, and two of which are new, cases 5 and 6.
    \item $\tilde{\mu}(b,b)=\{a,b,c\}$. Here, there are four cases, three of which are isomorphic (under exchange of $b$ and $c$) to cases 2, 4, and 6, and one of which is new, case 7.
    \item $\tilde{\mu}(b,b) = \emptyset$. In this case,     \begin{equation*}
    \widetilde{\mu \circ (\mu \times \idx)}(b,b,c) = \emptyset, 
    \end{equation*}
    whereas
    \begin{equation*}
    \widetilde{\mu \circ (\idx \times \mu)}(b,b,c) = \bigcup_{x \in \tilde{\mu}(b,c)} \tilde{\mu}(b,x),
    \end{equation*}
    which contains $b$ because $a \in \tilde{\mu}(b,c)$ and $b \in \tilde{\mu}(b,a)$. Thus associativity is contradicted and this case is ruled out.
    \item $\tilde{\mu}(b,b)=\{b\}$. This gives us cases 8 and 9.
    \item $\tilde{\mu}(b,b)=\{c\}$. This gives us cases 10 and 11 in Table \ref{tab:3elements_inv_2}.
    \item $\tilde{\mu}(b,b)=\{b,c\}$. Here, there are three cases, two of which are isomorphic (under exchange of $b$ and $c$ to cases 9 and 11, and one of which is new, case 12.
\end{itemize}

\subsubsection{The case $\varepsilon=\{b\}$}
In this case nondegeneracy restricts the remaining entries to the following scenario:
\vskip 3mm
    \begin{center}
    \begin{tabular}{ | c | c | c |} 
\hline
$\{a\}$  & $\{b\}$& $\{c\}$ \\ 
\hline
$\{b\}$ & does not include $b$ & does not include $b$ \\ 
\hline
$\{c\}$ & does not include $b$ & includes $b$\\ 
\hline
\end{tabular}
\end{center}
\vskip 3mm
As above, we can consider each possibility for $\tilde{\mu}(b,b)$ and see what restrictions on the other entries arise from associativity. This leads to cases 13--20 in Tables \ref{tab:3elements_inv_2} and \ref{tab:3elements_inv_3}.

\subsection{The case $\eta=\{a,b\}$}

In this case, unitality partially determines the multiplication table as follows.
\vskip 3mm
\begin{center}
    \begin{tabular}{ | c | c | c |} 
\hline
$\{a\}$  & $\emptyset$ & \phantom{\{c\}} \\ 
\hline
$\emptyset$ & $\{b\}$& \\ 
\hline
 &  &  \\ 
\hline
\end{tabular}
\end{center}
\vskip 3mm
Additionally, unitality requires that $\tilde{\mu}(a,c) \cup \tilde{\mu}(b,c) = \{c\}$ and $\tilde{\mu}(c,a) \cup \tilde{\mu}(c,b) = \{c\}$. In particular, $\tilde{\mu}(a,c)$, $\tilde{\mu}(b,c)$, $\tilde{\mu}(c,a)$, and $\tilde{\mu}(c,b)$ do not include $a$ or $b$.

Up to isomorphism, there are two possible counits: $\varepsilon = \{a,b\}$ and $\varepsilon = \{a,c\}$. We consider each case separately.

\subsubsection{$\varepsilon=\{a,b\}$}
In this case, nondegeneracy requires that $\tilde{\mu}(c,c)$ contain $a$ or $b$. Up to isomorphism, we can assume $b \in \tilde{\mu}(c,c)$.

We know that $\tilde{\mu}(b,c)$ is either $\{c\}$ or $\emptyset$. If $\tilde{\mu}(b,c) = \emptyset$, then 
    \begin{equation*}
    \widetilde{\mu \circ (\mu \times \idx)}(b,c,c) = \emptyset,
    \end{equation*}
    whereas
    \begin{equation*}
    \widetilde{\mu \circ (\idx \times \mu)}(b,c,c) = \bigcup_{x \in \tilde{\mu}(c,c)} \tilde{\mu}(b,x),
    \end{equation*}
which would include $b$ because $b \in \tilde{\mu}(c,c)$ and $b \in \tilde{\mu}(b,b)$. This contradicts associativity, so we conclude that $\tilde{\mu}(b,c) = \{c\}$. Similar calculations show that $\tilde{\mu}(c,b) = \{c\}$ and $\tilde{\mu}(a,c) = \tilde{\mu}(c,a) = \emptyset$.

Suppose $a \in \tilde{\mu}(c,c)$. Then
    \begin{equation*}
    \widetilde{\mu \circ (\mu \times \idx)}(a,c,c) = \emptyset,
    \end{equation*}
    whereas
    \begin{equation*}
    \widetilde{\mu \circ (\idx \times \mu)}(a,c,c) = \bigcup_{x \in \tilde{\mu}(c,c)} \tilde{\mu}(a,x),
    \end{equation*}
which would include $a$. We conclude that $a \notin \tilde{\mu}(c,c)$. 

This leaves two possibilities, cases 21 and 22 in Table \ref{tab:3elements_inv_3}. We can see that associativity does hold in these cases by recognizing that they are disjoint unions of the trivial $1$-element Frobenius object $\{a\}$ with cases $1$ and $2$ in Table \ref{fig:2elements}.

\subsubsection{$\varepsilon=\{a,c\}$}
In this case, nondegeneracy implies that $\tilde{\mu}(a,c) = \tilde{\mu}(c,a) = \emptyset$, that $\tilde{\mu}(b,c) = \tilde{\mu}(c,b) = \{c\}$, and that $\tilde{\mu}(c,c)$ does not have $a$ or $c$. This leaves two possibilities, cases 23 and 24, which are disjoint unions of the trivial $1$-element Frobenius object $\{a\}$ with cases $3$ and $4$ in Table \ref{fig:2elements}.

\subsection{The case $\eta = \{a,b,c\}$}
In this case, the unit axiom completely determines the multiplication table, and there is only one possible counit. This is case 25.

\begin{thm}\label{thm:classification3}
    Up to isomorphism, there are 25 Frobenius objects in $\rel$ with three elements. They are listed in Tables \ref{tab:3elements_inv_1}--\ref{tab:3elements_inv_3}. 
\end{thm}

\begin{remark}
For some of the three-element Frobenius objects in $\rel$, we can give conceptual explanations. Cases 10 and 17 correspond to the group $\Z_3$ with the two (up to isomorphism) possible choices of counit. Cases 21, 23, and 25 correspond to groupoids. Case 3 corresponds to the conjugacy classes of $S_3$, as discussed in Section \ref{subsec:Conjugacy}. But the majority of the cases do not arise from known general constructions, and the exhaustive classification that we have done was needed in order to find them all.
\end{remark}

\subsection{Calculation of partition functions}

All of the three-element Frobenius objects in $\rel$ are commutative, so we can compute the associated partition functions. By Lemma \ref{lemma:intersection}, we have that the partition function, as a logical proposition, is ``True'' for cases 1--12 and 21--25. Case 17 is the group $\Z_3$ with nontrivial counit, so from Proposition \ref{prop: partition_group} we have that the partition function is ``$g \equiv 1 \pmod{3}$''.

This leaves seven cases to check directly, but since they have similarities there are some calculations that can be done for all of them simultaneously. First, since they all have $\eta = \{a\}$ and $\varepsilon = \{b\}$, it follows that $\hat{\alpha}(a) = b$, $\hat{\alpha}(b) = a$, and $\hat{\alpha}(c) = c$. Second, for all of them we have from \eqref{eqn:S} that
\begin{equation*}
    \begin{split}
        \tilde{S}(a) &= \bigcup_{y \in X} \tilde{\mu}(\hat{\alpha}(y),y) \\
        &= \tilde{\mu}(b,a) \cup \tilde{\mu}(a,b) \cup \tilde{\mu}(c,c) \\
        &= \tilde{\mu}(c,c),
    \end{split}
\end{equation*}
where in the last step we have used the fact that $b \in \tilde{\mu}(c,c)$ in the cases we are considering. 

In cases 15, 16, 18, 19, and 20, where $a \in \tilde{\mu}(c,c)$, it follows that $\tilde{\mu}(c,c) \subseteq S^g \circ \eta$ for all $g \geq 1$, so $Z(\Sigma_g) = T$ for $g \geq 1$. Thus, as a logical proposition, the partition function is ``$g \geq 1$''.

For case 13, we have $\tilde{S}(a) = \tilde{\mu}(c,c) = \{b\}$, and we also calculate 
\[\tilde{S}(b) = \bigcup_{y \in X} \bigcup_{z \in \tilde{\mu}(y,b)} \tilde{\mu}(\hat{\alpha}(y),z) = \tilde{\mu}(\hat{\alpha}(a),b) = \tilde{\mu}(b,b) = \emptyset.\]
From this we have $S \circ \eta = \{b\}$ and $S^g \circ \eta = \emptyset$ for $g \geq 2$, so $Z(\Sigma_1) = T$ and $S(\Sigma_g) = F$ for $g \geq 2$. As a logical proposition, the partition function is ``$g=1$''.

For case 14, we have $\tilde{S}(a) = \tilde{\mu}(c,c) = \{b,c\}$, $\tilde{S}(b) = \emptyset$ as in case 13, and we also calculate
\[\tilde{S}(c) = \bigcup_{y \in X} \bigcup_{z \in \tilde{\mu}(y,c)} \tilde{\mu}(\hat{\alpha}(y),z) = \tilde{\mu}(b,b) \cup \tilde{\mu}(c,b) \cup \tilde{\mu}(c,c) = \{b,c\}.\]
From this we have $S^g\circ \eta = \{b,c\}$ for $g \geq 1$, so $S(\Sigma_g) = T$ for $g \geq 1$. As a logical proposition, the partition function is ``$g \geq 1$''.

These results are listed in the final columns of Tables \ref{tab:3elements_inv_1}--\ref{tab:3elements_inv_3}.

\section{Future directions}
A natural generalization is the full classification of Frobenius objects in $\rel$ of higher cardinality, and the computation of their corresponding topological invariants. As suggested in the previous section, the computations become more involved as cardinality increases, so developing code to solve the ``equations'' associated to the axioms would help to accomplish this objective. It would also be interesting to see if there exists analogue in $\rel$ of the Artin-Wedderburn classification theorem of semisimple algebras.

An ongoing work in progress \cite{CMS} intends to extend the characterization of Frobenius objects in $\rel$ and $\Span$ to higher categories. In particular, we plan to connect the simplicial description of various flavors of Frobenius objects with 2-Segal spaces \cite{stern:2segal}.

Since we can view Frobenius objects in $\rel$ as a generalization of groupoids, it could be interesting to see whether certain concepts in the theory of groupoids (such as Morita equivalence) can be extended.

As stated in the introduction, this project arose with the aim of better understanding TQFTs with values in the symplectic category. Our long-term goal is to characterize such structures and connect them to topological field theories such as the Poisson sigma model and Dijkgraaf-Witten theory.

\newpage
\begin{table}[!h]
\begin{tabular}{|c|c|c|c|c|}
\hline
\textbf{Case}&\textbf{Unit} & \textbf{Counit} & \textbf{Multiplication}& \textbf{Partition function} \\ \hline
& & & &\\

1& $\{a\}$& $\{a\}$ &
\begin{tabular}{ | c | c | c | } 
\hline
$\{a\}$  & $\{b\}$& $\{c\}$ \\ 
\hline
$\{b\}$ & $\{a\}$& $\{c\}$ \\ 
\hline
 $\{c\}$ & $\{c\}$& $\{a,b\}$ \\ 
\hline
\end{tabular} & True
\\
[5ex]
\hline
& & & &\\
2& $\{a\}$& $\{a\}$ &
\begin{tabular}{ | c | c | c | } 
\hline
 $\{a\}$  & $\{b\}$& $\{c\}$ \\ 
\hline
 $\{b\}$ & $\{a\}$& $\{c\}$ \\ 
\hline
$\{c\}$ & $\{c\}$& $\{a,b,c \}$ \\ 
\hline
\end{tabular}& True
\\
[5ex]
\hline
& & & &\\
3& $\{a\}$& $\{a\}$ &
\begin{tabular}{ | c | c | c | } 
\hline
 $\{a\}$  & $\{b\}$& $\{c\}$ \\ 
\hline
 $\{b\}$ & $\{a,b\}$& $\{c\}$ \\ 
\hline
 $\{c\}$ & $\{c\}$& $\{a,b \}$ \\ 
\hline
\end{tabular}
& True
\\
[5ex]
\hline
& & & &\\

4& $\{a\}$& $\{a\}$ &
\begin{tabular}{ | c | c | c | } 
\hline
 $\{a\}$  & $\{a\}$& $\{c\}$ \\ 
\hline
 $\{b\}$ & $\{a,b\}$& $\{c\}$ \\ 
\hline
 $\{c\}$ & $\{c\}$& $\{a,b,c \}$ \\ 
\hline
\end{tabular} 
& True
\\
[5ex]
\hline
& & & &\\

5& $\{a\}$& $\{a\}$ &
\begin{tabular}{ | c | c | c | } 
\hline
 $\{a\}$  & $\{b\}$& $\{c\}$ \\ 
\hline
 $\{b\}$ & $\{a,c\}$& $\{b,c\}$ \\ 
\hline
$\{c\}$ & $\{b,c\}$& $\{a,b\}$ \\ 
\hline
\end{tabular}
& True
\\
[5ex]
\hline
& & & &\\
6& $\{a\}$& $\{a\}$ &
\begin{tabular}{ | c | c | c | } 
\hline
 $\{a\}$  & $\{b\}$& $\{c\}$ \\ 
\hline
$\{b\}$ & $\{a,c\}$& $\{b,c\}$ \\ 
\hline
 $\{c\}$ & $\{b,c\}$& $\{a,b,c\}$ \\ 
\hline
\end{tabular}& True
\\
[5ex]
\hline
& & & &\\
7& $\{a\}$& $\{a\}$ &
\begin{tabular}{ | c | c | c | } 
\hline
 $\{a\}$  & $\{b\}$& $\{c\}$ \\ 
\hline
$\{b\}$ & $\{a,b,c\}$& $\{b,c\}$ \\ 
\hline
 $\{c\}$ & $\{b,c\}$& $\{a,b,c\}$ \\ 
\hline
\end{tabular}
& True
\\
[5ex]
\hline
& & & &\\

8& $\{a\}$& $\{a\}$ &
\begin{tabular}{ | c | c | c | } 
\hline
 $\{a\}$  & $\{b\}$& $\{c\}$ \\ 
\hline
 $\{b\}$ & $\{b\}$& $\{a,b,c\}$ \\ 
\hline
 $\{c\}$ & $\{a, b,c\}$& $\{b,c\}$ \\ 
\hline
\end{tabular}
& True
\\
[5ex]
\hline
& & & &\\

9& $\{a\}$& $\{a\}$ &
\begin{tabular}{ | c | c | c | } 
\hline
 $\{a\}$  & $\{b\}$& $\{c\}$ \\ 
\hline
 $\{b\}$ & $\{b\}$& $\{a,b,c\}$ \\ 
\hline
$\{c\}$ & $\{a, b,c\}$& $\{c\}$ \\ 
\hline
\end{tabular}
& True
\\
[5ex]
\hline

\end{tabular}
    \caption{Frobenius objects in $\rel$ with three elements and their partition functions, part 1.}
    \label{tab:3elements_inv_1}

\end{table}
\newpage
\begin{table}
\begin{tabular}{|c|c|c|c|c|}
\hline
\textbf{Case}&\textbf{Unit} & \textbf{Counit} & \textbf{Multiplication}& \textbf{Partition function} \\ \hline
& & & &\\

10& $\{a\}$& $\{a\}$ &
\begin{tabular}{ | c | c | c | } 
\hline
$\{a\}$  & $\{b\}$& $\{c\}$ \\ 
\hline
 $\{b\}$ & $\{c\}$& $\{a\}$ \\ 
\hline
 $\{c\}$ & $\{a\}$& $\{b\}$ \\ 
\hline
\end{tabular}
& True
\\
[5ex]
\hline
& & & &\\
11& $\{a\}$& $\{a\}$ & \begin{tabular}{ | c | c | c | } 
\hline
 $\{a\}$  & $\{b\}$& $\{c\}$ \\ 
\hline
 $\{b\}$ & $\{c\}$& $\{a,b\}$ \\ 
\hline
 $\{c\}$ & $\{a, b\}$& $\{b,c\}$ \\ 
\hline
\end{tabular}
& True
\\
[5ex]
\hline
& & & &\\
12& $\{a\}$& $\{a\}$ &
\begin{tabular}{ | c | c | c | } 
\hline
 $\{a\}$  & $\{b\}$& $\{c\}$ \\ 
\hline
 $\{b\}$ & $\{b,c\}$& $\{a,b,c\}$ \\ 
\hline
 $\{c\}$ & $\{a, b,c\}$& $\{b,c\}$ \\ 
\hline
\end{tabular}
& True
\\
[5ex]
\hline
& & & &\\

13& $\{a\}$& $\{b\}$ &
\begin{tabular}{ | c | c | c | } 
\hline
 $\{a\}$  & $\{b\}$& $\{c\}$ \\ 
\hline
 $\{b\}$ & $\emptyset$& $\emptyset$ \\ 
\hline
 $\{c\}$ & $\emptyset$& $\{b\}$ \\ 
\hline
\end{tabular}
& $g=1$
\\
[5ex]
\hline
& & & &\\

14& $\{a\}$& $\{b\}$ &
\begin{tabular}{ | c | c | c | } 
\hline
 $\{a\}$  & $\{b\}$& $\{c\}$ \\ 
\hline
 $\{b\}$ & $\emptyset$& $\emptyset$ \\ 
\hline
 $\{c\}$ & $\emptyset$& $\{b,c\}$ \\ 
\hline
\end{tabular}
& $g \geq 1$
\\
[5ex]
\hline
& & & &\\
15& $\{a\}$& $\{b\}$ &  \begin{tabular}{ | c | c | c | } 
\hline
 $\{a\}$  & $\{b\}$& $\{c\}$ \\ 
\hline
 $\{b\}$ & $\{a\}$& $\{c\}$ \\ 
\hline
 $\{c\}$ & $\{c\}$& $\{a,b\}$ \\ 
\hline
\end{tabular}& $g \geq 1$
\\
[5ex]
\hline
& & & &\\
16& $\{a\}$& $\{b\}$ & \begin{tabular}{ | c | c | c | } 
\hline
 $\{a\}$  & $\{b\}$& $\{c\}$ \\ 
\hline
 $\{b\}$ & $\{a\}$& $\{c\}$ \\ 
\hline
 $\{c\}$ & $\{c\}$& $\{a,b, c\}$ \\ 
\hline
\end{tabular}
& $g \geq 1$
\\
[5ex]
\hline
& & & &\\

17& $\{a\}$& $\{b\}$ &  \begin{tabular}{ | c | c | c | } 
\hline
 $\{a\}$  & $\{b\}$& $\{c\}$ \\ 
\hline
 $\{b\}$ & $\{c\}$& $\{a\}$ \\ 
\hline
 $\{c\}$ & $\{a\}$& $\{b\}$ \\ 
\hline
\end{tabular}
& $ g\equiv 1 (\mbox{mod  }3)$
\\
[5ex]
\hline
& & & &\\

18& $\{a\}$& $\{b\}$ &  \begin{tabular}{ | c | c | c | } 
\hline
 $\{a\}$  & $\{b\}$& $\{c\}$ \\ 
\hline
 $\{b\}$ & $\{c\}$& $\{a,c\}$ \\ 
\hline
 $\{c\}$ & $\{a, c\}$& $\{a,b,c\}$ \\ 
\hline
\end{tabular}
& $g \geq 1$
\\
[5ex]
\hline

\end{tabular}
    \caption{Frobenius objects in $\rel$ with three elements and their partition functions, part 2.}
    \label{tab:3elements_inv_2}

\end{table}
\newpage
\begin{table}[!h]
\begin{tabular}{|c|c|c|c|c|}
\hline
\textbf{Case}&\textbf{Unit} & \textbf{Counit} & \textbf{Multiplication}& \textbf{Partition function} \\ \hline
& & & &\\

19& $\{a\}$& $\{b\}$ &  \begin{tabular}{ | c | c | c | } 
\hline
 $\{a\}$  & $\{b\}$& $\{c\}$ \\ 
\hline
 $\{b\}$ & $\{a,c\}$& $\{a,c\}$ \\ 
\hline
 $\{c\}$ & $\{a,c\}$& $\{a,b\}$ \\ 
\hline
\end{tabular}
& $g \geq 1$
\\
[5ex]
\hline
& & & &\\
20& $\{a\}$& $\{b\}$ &  \begin{tabular}{ | c | c | c | } 
\hline
 $\{a\}$  & $\{b\}$& $\{c\}$ \\ 
\hline
 $\{b\}$ & $\{a,c\}$& $\{a,c\}$ \\ 
\hline
 $\{c\}$ & $\{a,c\}$& $\{a,b,c\}$ \\ 
\hline
\end{tabular}
& $g \geq 1$
\\
[5ex]
\hline
& & & &\\

21& $\{a,b\}$& $\{a,b\}$ &  \begin{tabular}{ | c | c | c | } 
\hline
 $\{a\}$  & $\emptyset$& $\emptyset$ \\ 
\hline
 $\emptyset$ & $\{b\}$& $\{c\}$ \\ 
\hline
 $\emptyset$ & $\{c\}$& $\{b\}$ \\ 
\hline
\end{tabular}
& True
\\
[5ex]
\hline
& & & &\\

22& $\{a,b\}$& $\{a,b\}$ &  \begin{tabular}{ | c | c | c | } 
\hline
 $\{a\}$  & $\emptyset$& $\emptyset$ \\ 
\hline
 $\emptyset$ & $\{b\}$& $\{c\}$ \\ 
\hline
 $\emptyset$ & $\{c\}$& $\{b,c\}$ \\ 
\hline
\end{tabular}
& True
\\
[5ex]
\hline
& & & &\\

23& $\{a,b\}$& $\{b,c\}$ &  \begin{tabular}{ | c | c | c | } 
\hline
 $\{a\}$  & $\emptyset$& $\emptyset$ \\ 
\hline
 $\emptyset$ & $\{b\}$& $\{c\}$ \\ 
\hline
 $\emptyset$ & $\{c\}$& $\{b\}$ \\ 
\hline
\end{tabular}
& True
\\
[5ex]
\hline
& & & &\\

24& $\{a,b\}$& $\{b,c\}$ &  \begin{tabular}{ | c | c | c | } 
\hline
 $\{a\}$  & $\emptyset$& $\emptyset$ \\ 
\hline
 $\emptyset$ & $\{b\}$& $\{c\}$ \\ 
\hline
 $\emptyset$ & $\{c\}$& $\emptyset$ \\ 
\hline
\end{tabular}
& True
\\
[5ex]
\hline
& & & &\\

25& $\{a,b,c\}$& $\{a,b,c\}$ &  \begin{tabular}{ | c | c | c | } 
\hline
 $\{a\}$  & $\emptyset$& $\emptyset$ \\ 
\hline
 $\emptyset$ & $\{b\}$& $\emptyset$ \\ 
\hline
 $\emptyset$ & $\emptyset$& $\{c\}$ \\ 
\hline
\end{tabular}& True
\\
[5ex]
\hline
\end{tabular}
    \caption{Frobenius objects in $\rel$ with three elements and their partition functions, part 3.}
    \label{tab:3elements_inv_3}

\end{table}

\clearpage
\bibliography{frob}
\end{document}